%% file: main.tex
\documentclass[article, onesided, a4paper]{amsart}

\input{preamble}

\title{Tressl's Structure Theorem for Separable Algebras}
\author{Gabriel Ng}
\date{10th February 2025}
\subjclass{Primary: 12H05, 13N99}
\keywords{differential algebras, structure theorem, differentially large fields, separable algebras}
\thanks{This project has received funding from the European Union’s Horizon 2020 research and innovation programme under the Marie Skłodowska Curie grant agreement No 101034383.}

\begin{document}

\begin{abstract}
    This short paper presents a generalisation of Tressl's structure theorem for differentially finitely generated algebras over differential rings of characteristic 0 to the case of separable algebras over differential rings of arbitrary characteristic.
\end{abstract}

\maketitle
\section{Introduction}

Tressl's structure theorem for differential algebras \cite[Theorem 1]{Tressl2002} is a result which describes the structure of differentially finitely generated algebras over differential rings of characteristic 0. It has been applied by Le\'on Sanchez and Tressl in the study of differentially large fields of characteristic 0 in order to characterise these fields in terms of the existence of points of differential algebras of a certain form \cite{SanchezTressl2020}.

In this paper, we adapt Tressl's proof to generalise this theorem to the case of separable algebras over differential rings of arbitrary characteristic. This generalised theorem is intended to play an analogous role in the study of differentially large fields in characteristic $p > 0$, for example in the development of a uniform companion for large differential fields in positive characteristic (cf. \cite{Tressl2005}), and various algebraic and geometric characterisations of differentially large fields in positive characteristic as in \cite{SanchezTressl2020} and \cite{SanchezTressl2023}.

\section{Preliminaries and Characteristic Sets}

We assume throughout this paper that all rings are commutative and unital, and of arbitrary characteristic. Recall that a \textbf{derivation} on a ring $R$ is an additive map $\delta: R \to R$ which satisfies the product rule $\delta(ab) = \delta(a)b + a\delta(b)$ for all $a, b \in R$. For the purposes of this paper, a \textbf{differential ring} is a ring $R$ equipped with an $m$-tuple $\bdelta = (\delta_0,...,\delta_{m-1})$ of \emph{commuting} derivations, i.e. for any $i \leq j < m$, $\delta_i\circ\delta_j = \delta_j \circ\delta_i$. For this section, we let $(R, \bdelta)$ be a differential ring.

We fix the following notation and objects:
\begin{itemize}
    \item $Y = (Y_0,...,Y_{n-1})$ is a tuple of $n$ indeterminates;
    \item $\mathcal{D} = \{\bdelta^\alpha : \alpha \in \N^m\}$ is the free abelian monoid of differential operators generated by $\bdelta$.
    \item $\mathcal{D}Y = \{\theta Y_i : \delta \in \mathcal{D}, i < n\}$ is the set of derivatives of variables from $Y$;
    \item $\mathcal{D}Y^* = \{(\theta Y_i)^p : p \in \N_{>0}, \theta Y_i \in \DY\}$ is the set of powers of elements from $\mathcal{D}Y$;
    \item $R\{Y\}$ is the \textbf{differential polynomial ring} in indeterminates $Y$, identified with $R[\mathcal{D}Y]$ as an $R$-algebra and equipped with the derivations $\bdelta$ extending those on $R$ and satisfying $\delta_i(\theta y_j) = (\delta_i\theta) y_j$ for all $\theta y_j \in \DY$ and $i < m$.
\end{itemize}

\begin{defn}
    We define the \textbf{rank} on $\mathcal{D}Y^*$ to be the map
    \begin{align}
        \rk: \mathcal{D}Y^* &\to \N \times n \times \N^m \times \N_{>0} \\
        (\bdelta^\alpha Y_i)^p &\mapsto (|\alpha|, i, \alpha_{m-1},...,\alpha_0, p).
    \end{align}
    We equip $\O = \N \times n \times \N^m \times \N_{>0}$ with the lexicographic order, which is well ordered.
\end{defn}

\begin{defn}
    For a differential polynomial $f \in R\{Y\} \setminus R$, say that a variable $y \in \DY$ \textbf{appears} in $f$ if $y$ appears in $f$ considered as an algebraic polynomial in $R[\DY]$. The \textbf{leader of $f$}, denoted $u_f$, is the variable $y \in \DY$ of maximal rank which appears in $f$. Denote by $u_f^*$ the highest power of $u_f$ which appears in $f$. We extend the notion of rank from $\DY^*$ to $R\{Y\}$ by setting
    \[
    \rk(f) = \rk(u_f^*) \in \O.
    \]
\end{defn}

\begin{defn}
    Let $f, g \in R\{Y\}$ be differential polynomials, with $g \not\in R$. We say that \textbf{$f$ is partially reduced with respect to $g$} (also called \textbf{weakly reduced}), if no proper derivative of $u_g$ appears in $f$. Say that \textbf{$f$ is reduced with respect to $g$} if $f$ is partially reduced with respect to $g$, and $\deg_{u_g}(f) < \deg_{u_g}(g)$.

    We say that $f$ is \textbf{(partially) reduced with respect to $G$}, where $G \subseteq R\{Y\}\setminus R$, if $f$ is (partially) reduced with respect to every $g \in G$.

    A nonempty subset $G \subseteq R\{Y\}$ is \textbf{autoreduced} if every $f \in G$ is reduced with respect to every $g \neq f \in G$ (including the case where $G$ is empty).
\end{defn}

It is easy to see that if $G$ is autoreduced, then for any distinct $f, g \in G$, $u_f \neq u_g$.

\begin{prop}
    Every autoreduced set is finite. 
\end{prop}
\begin{proof}
    Suppose that $G$ is an infinite autoreduced set. For each $f \in G$, write $u_f = \delta_0^{\alpha_{0, f}}...\delta_{m-1}^{\alpha_{m-1, f}} Y_{i_f}$. Consider the set
    \[
    \{(\alpha_{0, f},...,\alpha_{m-1, f}, i_f): f \in G\} \subseteq \N^m \times n
    \]
    partially ordered by the product order. Then, by \cite[Chapter 0, Section 17, Lemma 15(a)]{Kolchin1973}, there exists an infinite increasing sequence of elements, where every term has the same projection onto $n$. This implies there are $f_0, f_1, f_2,...$ such that $u_{f_j}$ is a proper derivative of $u_{f_i}$ for all $i < j$, which contradicts the assumption that $G$ is autoreduced.
\end{proof}

\begin{defn}
    Let $\infty$ be an element larger than every element of $\O$, and equip $(\O \cup \{\infty\})^\N$ with the lexicographic order.  The \textbf{rank of an autoreduced set} $G$ is defined as follows: Let $G = \{g_0,...,g_{l-1}\}$ with $\rk(g_0) < ... < \rk(g_{l-1})$. Define
    \[
    \rk(G) = (\rk(g_0),... , \rk(g_{l-1}), \infty, \infty,...) \in \O.
    \]
\end{defn}

\begin{prop}[{\cite[Chapter I, Section 10, Proposition 3]{Kolchin1973}}]
    There is no infinite strictly rank-decreasing sequence of autoreduced sets.
\end{prop}

\begin{defn}[{cf. \cite[Chapter I, Section 10]{Kolchin1973}}]
    Let $\mathfrak{k} \subseteq R\{Y\}$ be a differential ideal not contained in $R$. Then, by the previous proposition,
    \[
    \{\rk(G) : G \subseteq M \text{ is autoreduced such that } S(g) \not\in \mathfrak{k} \text{ for all } g \in G\}
    \]
    has a minimum. We call an autoreduced subset $G$ of $M$ of minimal rank a \textbf{characteristic set of $\mathfrak{k}$}.
\end{defn}

\begin{lem} \label{lem_char_set_not_reduced}
    Let $G$ be a characteristic set of $M \subseteq R\{Y\}$, and $f \in M\setminus R$. Then, $f$ is not reduced with respect to $G$. 
\end{lem}

\begin{defn} 
    Let $f \in R\{Y\} \setminus R$, and write $f = f_d u_f^d + ... + f_1 u_f + f_0$, where $f_i \in R[y \in \DY, y \neq u_f]$, and $f_d \neq 0$. The \textbf{initial of $f$}, denoted $I(f)$ is:
    \[
    I(f) = f_d.
    \]
    The \textbf{separant of $f$}, denoted $S(f)$ is:
    \[
    S(f) = \frac{\partial}{\partial u_f} f = df_du_f^{d-1} + ... + f_1.
    \]
    Finally, for any autoreduced subset $G = \{g_0,...,g_{l-1}\}$ of $R\{Y\}$, define
    \[
    H(G) = \prod_{i < l} I(g_i)\cdot S(g_i),
    \]
    i.e. $H(G)$ is the product of all initials and separants of elements of $G$, and
    \[
    H_G = \left\{ \prod_{i < l} I(g_i)^{n_i} S(g_i)^{m_i} : n_i, m_i \in \N \right\},
    \]
    i.e. $H_G$ is the set of all products of powers of initials and separants of elements of $G$.
\end{defn}

\begin{prop}[{\cite[Chapter I, Section 9, Proposition 1]{Kolchin1973}}]\label{prop_kolchin_division}
    Let $G$ be an autoreduced set in $R\{Y\}$ and $f \in R\{Y\}$. There exists $\tilde{f} \in R\{Y\}$ which is reduced with respect to $G$, and $H \in H_G$ such that
    \[
    H \cdot f \equiv \tilde{f} \mod{[G]}.
    \]
    If in addition $f$ is partially reduced with respect to $G$, then there is $h \in H_G$ such that
    \[
    H \cdot f \equiv \tilde{f} \mod{(G)}.
    \]
\end{prop}

\section{(Quasi-)separable Rings and Ideals}

We recall the notions of separable and quasi-separable rings and ideals from \cite[Chapter 0]{Kolchin1973}. 

\begin{defn}
    Let $L/K$ be a field extension. We say that a family of elements $(a_i)_{i \in I}$ of $L$ is \textbf{separably dependent over $K$} if there is a polynomial $f \in K[x_i : i \in I]$ which vanishes at $(a_i)$, and at least one of the partial derivatives $\frac{\partial f}{\partial x_i}$ does not vanish at $(a_i)$.
\end{defn}

\begin{fact}[{\cite[Chapter 0, Section 3]{Kolchin1973}}]
    We observe the following: let $L/K$ be a field extension as above. For an arbitrary family $(a_i)_{i \in I}$ of elements of $L$, if there is a subset $J \subseteq I$ for which $(a_i)_{i \in J}$ is a transcendence basis for $K(a_i : i \in I)$ over $K$ such that $I \setminus J$ is finite, then the cardinality $r = |I \setminus J|$ is independent of the choice of $J$, and $r$ is called the \textbf{algebraic codimension} of $(a_i)$ over $K$, and say that $(a_i)$ has \textbf{finite algebraic codimension over $K$}.
\end{fact}

\begin{defn}
    A field extension $L/K$ (not necessarily algebraic) is \textbf{separable} (respectively, \textbf{quasi-separable}) if every family of elements of $L$ which is separably independent over $K$ has algebraic codimension $0$ (respectively, finite algebraic codimension).
\end{defn}

\begin{defn}
    Let $R \subseteq S$ be rings, where $S$ is a domain. We say that $S$ is \textbf{separable} (respectively, \textbf{quasi-separable}) over $R$ if $Q(S)$ is a separable (respectively, \textbf{quasi-separable}) field extension of $Q(R)$, where $Q(S)$ and $Q(R)$ denote the quotient fields of $S$ and $R$, respectively.

    Let $\mathfrak{k} \subset S$ be a prime ideal, and $f$ denote the quotient map $S \to S/\mathfrak{k}$. We say that $\mathfrak{k}$ is \textbf{separable} (respectively, \textbf{quasi-separable}) over $R$ if $f(S)$ is separable (respectively, quasi-separable) over $f(R)$.
\end{defn}

\begin{rmk}
    In particular, we note the special case when $L/K$ is a separable field extension, any $K$-subalgebra of $L$ is a separable $K$-algebra.
\end{rmk}

From now, let $S = R\{Y\}$ be a differentially finitely generated differential polynomial algebra over $R$. For a set $A \subseteq S$, denote the differential ideal generated by $A$ by $[A]$. For an ideal $\mathfrak{k} \subseteq S$ and $s \in S$, define the \textbf{saturated ideal of $\mathfrak{k}$ over $s$} as
\[
    \mathfrak{k} : s^\infty = \{h \in S : s^n h \in \mathfrak{k} \text{ for some } n \in \N\}
\]
It is easy to verify by direct computation that if $\mathfrak{k}$ is a differential ideal, then $\mathfrak{k}:s^\infty$ is also differential.

\begin{defn}[{\cite[Chapter III, Section 8]{Kolchin1973}}]
    Let $G$ be an autoreduced set in $S$, and $\mathfrak{k}$ be an ideal (not necessarily differential) of $S$. We say that $G$ is  \textbf{$\mathfrak{k}$-coherent} if the following hold:
    \begin{enumerate}
        \item $\mathfrak{k}$ has a set of generators partially reduced with respect to $G$;
        \item $[\mathfrak{k}] \subseteq ([G] + \mathfrak{k}) : H(G)^\infty$
        \item For any $f, g \in G$, and $v$ a common derivative of $u_f, u_g$, say $v = \theta_f u_f = \theta_g u_g$, we have that
        \[
        S(g)\theta_f f - S(f)\theta_g g \in ((G_v) + \mathfrak{k}) : H(G)^\infty,
        \]
        where $G_v$ denotes the set of all differential polynomials of the form $\tau h$, where $h \in G$, $\tau \in \Theta$, and $\tau u_h$ is of lower order than $v$.
    \end{enumerate}
\end{defn}

\begin{lem}[{\cite[Chapter III, Section 8, Lemma 3]{Kolchin1973}}] \label{lem_coherent_autoreduced_finiteness}
    Let $\mathfrak{p}$ be a prime differential ideal of $S$ which is quasi-separable over $R$, and let $G$ be a characteristic set of $\mathfrak{p}$. Then, there exists a finite set $Y' \subseteq \DY$ of derivatives of the indeterminates $Y$, each partially reduced with respect to $G$, such that by setting $\mathfrak{p}_1 = \mathfrak{p} \cap R[Y']$, $G$ is $S\mathfrak{p}_1$-coherent and $\mathfrak{p} = ([G] + S\mathfrak{p}_1) : H(G)^\infty$. The set $Y'$ may be replaced by any larger finite set of derivatives of $Y$ partially reduced with respect to $A$.
\end{lem}

\begin{rmk}
    Recall that by the definition of a characteristic set, for each $a \in G$, $S(a) \not\in \mathfrak{p}$. Further, for any $g \in G$, $I(g) \not\in \mathfrak{p}$ by \cite[Chapter I, Section 10, Lemma 8]{Kolchin1973}. Since $\mathfrak{p}$ is a prime ideal, then we also have that $H_G$ does not contain any element of $\mathfrak{p}$, as for any $g \in G$, neither $I(g)$ nor $S(g)$ lie in $\p$.
\end{rmk}

We will also require an additional lemma which strengthens the above result for partially reduced polynomials:

\begin{lem}[{\cite[Chapter III, Section 8, Lemma 5]{Kolchin1973}}]
    Let $G$ be a $\mathfrak{k}$-coherent autoreduced set in $S$, and suppose that for each $g \in G$, $S(g)$ is not a zero-divisor in $S$. Then, every element of $([G] + \mathfrak{k}) : H(G)^\infty$ which is partially reduced with respect to $G$ lies in $((G) + \mathfrak{k}) : H(G)^\infty$.
\end{lem}

In particular, in the case where $G$ is a characteristic set of a prime differential ideal $\p \subseteq S$, we have the following:

\begin{cor} \label{cor_coherent_partially_reduced}
    Let $\p$, $G$ and $\p_1$ be as in Lemma \ref{lem_coherent_autoreduced_finiteness}. Then, any $f \in \p$ partially reduced with respect to $G$ lies in $((G) + S\p_1) : H(G)^\infty$.
\end{cor}

\section{The Structure Theorem for Separable Algebras}

We state and prove the main theorem, which is an adapted form of \cite[Theorem 1]{Tressl2002}:

\begin{thm}\label{str_thm_sep}
    Let $(S, \bpartial)$ be a differential domain, and let $(R, \bdelta)$ be a differential subring of $S$, with $R$ Noetherian as a ring, such that:
    \begin{enumerate}
        \item $S$ is separable over $R$; and,
        \item $S$ is differentially finitely generated as a differential $R$-algebra.
    \end{enumerate}
    Then, there exist (not necessarily differential) $R$-subalgebras $P$ and $B$ of $S$ and an element $h \in B$ with $h \neq 0$ such that:
    \begin{enumerate}[label=(\alph*)]
        \item $B$ is a finitely generated $R$-algebra, and $B_h$ is a finitely presented $R$-algebra;
        \item $P$ is a polynomial algebra over $R$;
        \item $S_h = (B\cdot P)_h$, and $S_h$ is a differentially finitely presented $R$-algebra;
        \item The homomorphism $B \otimes_R P \to B \cdot P$ induced by multiplication is an isomorphism of $R$-algebras.
    \end{enumerate}
\end{thm}

\begin{rmk}
    The assumption that $R$ is Noetherian does not appear in the theorem in \cite{Tressl2002}. We use this to show the finite-presentedness of $B_h$ and $S_h$, as Tressl's proof in characteristic 0 relies on Corollary 6 of \cite{Tressl2002}, which does not apply in the case of arbitrary characteristic. As we usually intend for $R$ to be a differential field, this is a relatively mild assumption. 
\end{rmk}

\subsection*{Proof of Theorem \ref{str_thm_sep}}
    As $S$ is a differentially finitely generated $R$-algebra, there is a surjective differential $R$-algebra homomorphism $\phi: R\{Y_1,...,Y_n\} \to S$. Write $Y = (Y_1,...,Y_n)$, and let $\mathfrak{p} = \ker(\phi)$. As $S$ is a domain, and $\phi$ is injective on $R$, we have that $\mathfrak{p}$ is a differential prime ideal of $R\{Y\}$ with $\mathfrak{p} \cap R = 0$. Further, since $S$ is separable over $R$, the ideal $\mathfrak{p}$ is separable over $R$. Let $G$ be a characteristic set of $\mathfrak{p}$.

    By Lemma \ref{lem_coherent_autoreduced_finiteness}, there is a finite set $Y' \subseteq \DY$, whose members are each partially reduced with respect to $G$, such that $G$ is $S\mathfrak{p}_1$-coherent, and $\mathfrak{p} = ([G] + S\mathfrak{p}_1) : H(G)^\infty$, where $\mathfrak{p}_1 = \mathfrak{p} \cap R[Y']$.

    Define the following:
    \begin{align}
        h &\coloneqq \phi(H(G)), \\
        V &\coloneqq \{y \in \DY : y \text{ is not a proper derivative of any } u_g \in G\}, \\
        V_B &\coloneqq Y' \cup \{y \in V : y \text{ appears in some } g \in G\}, \\
        B &\coloneqq \phi(R[V_B]), \\
        P &\coloneqq \phi(R[V \setminus V_B]).
    \end{align}

    \begin{rmk}
        Observe that $Y'$ is a subset of $V$: if a differential indeterminate is a proper derivative of some $u_g$, then it is not partially reduced with respect to $G$ and thus does not lie in $Y'$.
    \end{rmk}
    \begin{rmk}
        Since $G$ is autoreduced, $f \in R\{Y\}$ is partially reduced with respect to $G$ if and only if $f \in R[V]$.
    \end{rmk}

    \begin{claim_num}\label{claim_restriction_not_vb_injective}
        The restriction of $\phi$ to $P' = R[V \setminus V_B]$ is injective.
    \end{claim_num}
    \begin{proof}
        Suppose we have $f \in P' \cap \p$. We claim that $f$ is reduced with respect to $G$. By the previous observation, we have that $f$ is partially reduced with respect to $G$. Since the leader $u_g$ of each $g \in G$ is in $V_B$, no leader of any $g \in G$ appears in $f$, and thus $f$ is reduced with respect to $G$. By Lemma \ref{lem_char_set_not_reduced}, $f \in R$, and since $\p \cap R = 0$, $f = 0$, as required.
    \end{proof}

    \begin{claim_num}
        $h \neq 0$.
    \end{claim_num}
    \begin{proof}
        By the remark following Lemma \ref{lem_coherent_autoreduced_finiteness} and writing $G = \{g_0,...,g_{l-1}\}$, we observe that
        \[
        H(G) = \prod_{i < l} I(g_i)\cdot S(g_i)
        \]
        does not lie in the prime ideal $\p$, as no $I(g_i)$ nor $S(g_i)$ lies in $\p$. Thus $h = \phi(H(G))$ is not $0$ in $S$.
    \end{proof}

    \begin{claim_num}
         $S_h = (B \cdot P)_h$.
    \end{claim_num}
    \begin{proof}
        Let $f \in R\{Y\}$. By Proposition \ref{prop_kolchin_division}, there is $\tilde{f} \in R\{Y\}$ reduced with respect to $G$, and $H \in H_G$ such that $H \cdot f \equiv \tilde{f} \mod{[G]}$. Since $G \subseteq \p$, we have that $[G] \subseteq \p$, thus $\phi(f) \cdot \phi(H) = \phi(\tilde{f})$ in $S$.
        
        Since $\tilde{f}$ is reduced with respect to $G$, we have that $\tilde{f} \in R[V]$, and in particular $\phi(\tilde{f}) \in B \cdot P$.
        Further, for each $g \in G$, $\phi(I(g))$ and $\phi(S(g))$ are units in $(B \cdot P)_h$, so $\phi(H)$ is a unit also. Thus, $\phi(f) = \phi(H)^{-1} \phi(\tilde{f}) \in (B \cdot P)_h$, as required. The reverse inclusion is trivial.
    \end{proof}

    \begin{claim_num}
        $S_h$ is a differentially finitely presented $R$-algebra.
    \end{claim_num}
    \begin{proof}
        We extend the projection $\phi: R\{Y\} \to S$ to $\psi: R\{Y\}[H(G)^{-1}] \to S_h$ by defining $\psi(H(G)^{-1}) = h^{-1}$. Thus $S_h$ is differentially finitely generated.

        Let $\mathfrak{q} = \ker(\psi)$. We show that $\mathfrak{q}$ is finitely generated as a differential ideal. As $R$ is Noetherian, and $R[Y']$ is a finitely generated polynomial algebra over $R$, it is also Noetherian, and thus $\p_1$ is finitely generated as an ideal. Let $A$ be such an generating set.
        Clearly, since $\psi$ extends $\phi$, $G \cup A \subseteq \mathfrak{q}$. We claim that $G \cup A$ generates $\mathfrak{q}$ as a differential ideal. Since $\psi$ is an extension of $\phi$, we have that $\mathfrak{p} \subset \mathfrak{q}$, therefore $[G \cup A] \subseteq \mathfrak{q}$ holds.
        
        For the reverse inclusion, suppose that $f/H(G)^d \in \mathfrak{q}$ for some $f \in R\{Y\}$ and $d \in \N$.
        Since $S_h$ is a domain, and $h^{-d}$ is nonzero, we have that $\psi(f) = \phi(f) = 0$, i.e. $f \in \p$. Since $\p = ([G] + S\p_1) : H(G)^\infty$, there is some $n \in \N$ with $fH(G)^n \in [G] + S\p_1 \subseteq [G \cup A]$. Thus, we have that $f$ lies in the differential ideal generated by $G \cup A$ in $R\{Y\}[H(G)^{-1}]$.
    \end{proof}

    \begin{claim_num}
        $B$ is finitely generated, and $B_h$ is finitely presented as $R$-algebras.
    \end{claim_num}
    \begin{proof}
        This is clear, as $B$ is a quotient of a finitely generated polynomial algebra over a Noetherian ring.
    \end{proof}

    The above claims prove parts (a), (b) and (c) of Theorem \ref{str_thm_sep}. It remains to prove (d).

    \begin{claim_num}\label{claim_linear_dependence}
        Suppose $b_1,...,b_m \in B$ are linearly dependent over $P$. Then, they are linearly dependent over $R$.
    \end{claim_num}
    \begin{proof}
        For each $i$, let $f_i \in R[V_B]$ with $\phi(f_i) = b_i$, and suppose that there are $p_i \in R[V \setminus V_B]$, not all in $\p$, such that $q \coloneqq \sum_i f_i p_i \in \p$. We may assume in particular that $p_1 \not\in \p$. Since $q \in R[V]$, $q$ is partially reduced with respect to $G$.

        Let $A$ again be a finite generating set of $\p_1$ as an ideal of $R[Y']$. Then, by Corollary \ref{cor_coherent_partially_reduced}, there are $n \in \N$, and $h_g \in R\{Y\}$ for each $g \in G \cup A$ such that
        \[
        H(G)^n q = \sum_{g \in G \cup A} h_g g.
        \]
        Since $H(G)$, $q$ and every $g$ lies in $R[V]$, we may assume the coefficients $h_g$ also lie in $R[V]$. 

        Since $p_1 \neq 0$, there exists an $R$-algebra homomorphism $\psi: R[V \setminus V_B] \to R$ such that $\psi(p_1) \neq 0$. This can be extended to an $R[V_B]$-algebra homomorphism $\psi: R[V] \to R[V_B]$ with $\psi(p_1) \neq 0$.

        Since all $p_i$ and $h_g$ lie in $R[V]$, we may now apply $\psi$ to the following equation:
        \[
        H(G)^n( f_1 p_1 + ... + f_k p_k) = \sum_{g \in G\cup A} h_g g.
        \]
        Since $H(G)$, every $f_i$ and $g \in G \cup A$ lies in $R[V_B]$, these are preserved by $\psi$, and thus we obtain that
        \[
        H(G)^n(\psi(p_1) f_1 + ... + \psi(p_k) f_k) = \sum_{g\in G\cup A} \psi(h_g) g.
        \]
        Observe that the right hand side lies in the ideal $([G] \cup S\p_1)$, and thus
        \[
        \psi(p_1) f_1 + ... + \psi(p_k) f_k \in ([G] \cup S\p_1) : H(G)^\infty = \p.
        \]
        Recall that $\psi(p_i) \in R$, and $\phi$ is an $R$-algebra homomorphism. Applying $\phi$ to the above yields:
        \[
        \psi(p_1) b_1 + ... + \psi(p_k) b_k = 0.
        \]
        By previous assumption, we have that $\psi(p_1)$ is not zero, and thus the $b_i$ are linearly dependent over $R$.
    \end{proof}
    
    This implies (d) as follows:
    \begin{claim_num}
        The homomorphism $B \otimes_R P \to B\cdot P$ induced by multiplication is an isomorphism of $R$-algebras.
    \end{claim_num}
    \begin{proof}
        The homomorphism is clearly surjective, and it remains to show injectivity. Let $m \in \N$ be minimal such that there are $b_1,...,b_m \in B$ and $p_1,...,p_m \in P$ with $\sum_i b_i p_i = 0$ and $x \coloneqq \sum_i b_i \otimes p_i \neq 0$. Then, the $b_i$ are linearly dependent over $P$. By Claim \ref{claim_linear_dependence}, the $b_i$ are in fact linearly dependent over $R$. Thus, there are $r_1,...,r_m \in R$, not all zero, such that $r_1b_1 + ... + r_mb_m = 0$.

        Without loss, assume that $r_1 \neq 0$. Then, $m > 1$, and observe that
        \[
        r_1b_1 \otimes p_1 = -(r_2b_2 + ... + r_mb_m) \otimes p_1,
        \]
        which gives
        \[
        r_1x = -(r_2b_2 + ... + r_mb_m) \otimes p_1 + r_1b_2 \otimes p_2 + ... + r_1b_m \otimes p_m.
        \]
        Collecting terms and rearranging, we have
        \[
        r_1x = b_2 \otimes (r_1p_2 - r_2p_1) + ... + b_m \otimes (r_1p_m - r_mp_1).
        \]
        By assumption, the image of $x$ under the homomorphism induced by multiplication is $0$, hence the image of $r_1x$ is also $0$. Thus, we necessarily have that $r_1x = 0$, otherwise this contradicts the minimality of the choice of $m$. Let $F$ denote the quotient field of $R$. Then, in $F \otimes_R (B \otimes_R P)$, we have that
        \[
        1 \otimes x = \frac{1}{r_1} \otimes r_1x = 0.
        \]
        By Claim \ref{claim_restriction_not_vb_injective}, $P$ is a polynomial $R$-algebra, hence flat. As the inclusion $B \to F \otimes_R B$ is injective, it follows that $B \otimes_R P \to F \otimes_R B \otimes_R P$ is injective. Thus, $1 \otimes x = 0$ in $F \otimes_R B \otimes_R P$ implies that $x = 0$, which is a contradiction.
    \end{proof}
    
\bibliography{references1}
\bibliographystyle{halpha}

\end{document}

%% file: preamble.tex
\usepackage[utf8]{inputenc}

\usepackage{amsmath}
\usepackage{amssymb}

\usepackage{bm}

\usepackage{physics}

\usepackage{xurl}

\usepackage{imakeidx}

\usepackage{hyperref}

\let\oldproof\proof
\def\proof{\oldproof\unskip}

\usepackage{etoolbox}

\makeatletter

\patchcmd\deferred@thm@head
  {\addvspace{-\parskip}}
  {}
  {}{\typeout{\string\deferred@thm@head patch failed!}}

\makeatletter 

\usepackage{mathtools}
\mathtoolsset{showonlyrefs}
\usepackage{enumitem}
\usepackage{amsthm}

\usepackage{tikz-cd}

\newtheorem{thm}[subsection]{Theorem}
\newtheorem{lem}[subsection]{Lemma}
\newtheorem{prop}[subsection]{Proposition}
\newtheorem{cor}[subsection]{Corollary}
\newtheorem{fact}[subsection]{Fact}

\newtheorem{claim_num}[subsection]{Claim}

\theoremstyle{remark}
\newtheorem*{rmk}{Remark}

\theoremstyle{definition}
\newtheorem{defn}[subsection]{Definition}

\DeclareMathOperator{\rk}{rk}

\newcommand{\bdelta}{\bm{\delta}}
\newcommand{\bpartial}{\bm{\partial}}

\newcommand{\DY}{\mathcal{D}Y}

\renewcommand{\leq}{\leqslant}

\renewcommand{\epsilon}{\varepsilon}
\renewcommand{\O}{\mathcal{O}}
\renewcommand{\phi}{\varphi}

\newcommand{\N}{\mathbb{N}}

\newcommand{\p}{\mathfrak{p}}

\def\Ind#1#2{#1\setbox0=\hbox{$#1x$}\kern\wd0\hbox to 0pt{\hss$#1\mid$\hss}
\lower.9\ht0\hbox to 0pt{\hss$#1\smile$\hss}\kern\wd0}

\def\notind#1#2{#1\setbox0=\hbox{$#1x$}\kern\wd0
\hbox to 0pt{\mathchardef\nn=12854\hss$#1\nn$\kern1.4\wd0\hss}
\hbox to 0pt{\hss$#1\mid$\hss}\lower.9\ht0 \hbox to 0pt{\hss$#1\smile$\hss}\kern\wd0}

\usepackage{graphicx}
